\documentclass{article}[12pt]

\usepackage{amsmath}
\usepackage{amssymb}
\usepackage{amsthm}
\usepackage{enumerate}
\usepackage{hyperref}

\newtheorem{theorem}{Theorem}
\newtheorem{prop}[theorem]{Proposition}
\newtheorem{lemma}[theorem]{Lemma}
\newtheorem{cor}[theorem]{Corollary}
\newtheorem{conj}[theorem]{Conjecture}
\newtheorem{definition}{Definition}[section]

\newcommand{\ZZ}{\mathbb{Z}}
\newcommand{\NN}{\mathbb{N}}
\newcommand{\WW}{\mathbb{W}}
\newcommand{\FF}{\mathbb{F}}

\allowdisplaybreaks
\date{}
\begin{document}
\title{On the Reciprocal of the Binary Generating Function for the Sum of Divisors}
\author{\Large{Joshua N. Cooper \thanks{This work is funded in part by NSF grant DMS-1001370.}} \\ Department of Mathematics \\ University of South Carolina \\ Columbia, SC 29208 \\ USA \\ \href{mailto:cooper@math.sc.edu}{cooper@math.sc.edu} \and \Large{Alexander W.~N.~Riasanovsky} \\ Department of Mathematics \\ University of Pennsylvania \\ Philadelphia, PA 19104 \\ USA \\ \href{mailto:alexneal@sas.upenn.edu}{alexneal@sas.upenn.edu}}
\maketitle
\begin{abstract}
If \( A \) is a set of natural numbers containing \( 0 \), then there is a unique nonempty ``reciprocal'' set \( B \) of natural numbers (containing \( 0 \)) such that every positive integer can be written in the form \( a + b \), where \( a \in A \) and \( b \in B \), in an even number of ways.  Furthermore, the generating functions for \( A \) and \( B \) over \( \FF_2 \) are reciprocals in \( \FF_2 [[q]] \).  We consider the reciprocal set \( B \) for the set \( A \) containing \( 0 \) and all integers such that \( \sigma(n) \) is odd, where \( \sigma(n) \) is the sum of all the positive divisors of \( n \). This problem is motivated by Euler's ``Pentagonal Number Theorem'', a corollary of which is that the set of natural numbers \( n \) so that the number \( p(n) \) of partitions of an integer \( n \) is odd is the reciprocal of the set of generalized pentagonal numbers (integers of the form \( k(3k\pm1)/2 \), where \( k \) is a natural number).  An old (1967) conjecture of Parkin and Shanks is that the density of integers \( n \) so that \( p(n) \) is odd (equivalently, even) is  \(1/2 \).  Euler also found that  \( \sigma(n) \) satisfies an almost identical recurrence as that given by the Pentagonal Number Theorem, so we hope to shed light on the Parkin-Shanks conjecture by computing the density of the reciprocal of the set containing the natural numbers with \( \sigma(n) \) odd (\( \sigma(0)=1 \) by convention).  We conjecture this particular density is \( 1/32 \) and prove that it lies between \( 0 \) and \( 1/16 \).  We finish with a few surprising connections between certain Beatty sequences and the sequence of integers \( n \) for which \( \sigma(n) \) is odd.
\end{abstract}

\section{Introduction} \label{intro}
For any sets containing nonnegative integers \( A \) and \( B \), the asymmetric additive representation function is defined by
\[
R(n) = \left| \left\{(a,b)\in A\times B \mid n = a+b\right\} \right|
\]
Alternatively, we can define \( R(n) \) with
\[
\left( \sum_{a \in A} q^a \right)\left( \sum_{b \in B} q^b \right) = \sum_{n = 0}^\infty R(n)q^n
\]
We are interested in the case where \( R(n) \equiv 0 \pmod{2} \) for \( n \geq 1 \) and \( R(0) = 1 \) which is illustrated in the power series ring \( \FF_2  [[ q ]] \).  Here, \( A \) and \( B \) are called reciprocals.  For a set \( A \), we write its reciprocal as \( \overline{A} \). Given \( A \), we focus on the relative density of \( \overline{A} \),
\[
\delta \left( \overline{A},n\right) = \frac{\left|\overline{A}\cap\left[0,n\right]\right|}{n+1}
\]
and natural density \( \delta \left(\overline{A}\right) = \lim_{n \rightarrow \infty} \delta \left(\overline{A},n\right) \).  Consider the following statement of Euler's Pentagonal Number Theorem:
\begin{equation} \label{eq:EPNTZ}
\left ( \sum_{n = -\infty}^{\infty} (-1)^n q^{\frac{n(3n-1)}{2}}\right ) \left ( \sum_{n = 0}^\infty p(n)q^n \right ) = 1,
\end{equation}
where \( p(n) \) is the partition function of \( n \), i.e., the number of ways to write \( n \) as an unordered sum of positive integers.  If we rewrite (\ref{eq:EPNTZ}) mod \( 2 \), the result is
\[
\left ( \sum_{n = -\infty}^{\infty} q^{\frac{n(3n-1)}{2}}\right ) \left ( \sum_{n= 0}^\infty p(n) q^n \right ) = 1,
\]
where the power series are now elements of the ring \( \FF_2 [[q]] \).  In this sense, \( P \), the set of integers with an odd number of partitions (including \( 0 \)), is the reciprocal of the set of generalized pentagonal numbers \( G \), i.e., \( \overline{G} = P \).  A well-known and difficult conjecture of Parkin and Shanks states that \( \delta(P) = 1/2 \).  The current best lower bounds on the density of \( P \) still tend towards \( 0 \) \cite{A, NRS}.

The paper which precedes this attempts to shed light on the question by studying reciprocals mod \( 2 \) in general.  In particular, the authors found that a loosely-defined ``typical'' reciprocal set has density \( 1/2 \) \cite{CEO}.  With this in mind, we continue the line of work by studying the analogous reciprocal for the function \( \sigma(n) \), the sum-of-divisors function defined as
\[
\sigma(n) =\sum_{d|n}d.
\]
The motivating connection between \( p(n) \) and \( \sigma(n) \) is the fact that they satisfy almost identical recurrences \cite{B}
\[
p(n) = \sum_{n = -\infty}^{\infty}(-1)^n p\left(n-\frac{k(3k-1)}{2}\right)
\]
and
\[
\sigma(n) = \sum_{n = -\infty}^{\infty}(-1)^n\sigma\left(n-\frac{k(3k-1)}{2}\right)
\]
with the only difference being that \( \sigma (n-n) \) is interpreted to mean \( n \) (since \( \sigma(0) \) is undefined, whereas \( p(0)=1 \)). The reciprocal sets \( G \) and \( P = \overline{G} \) have densities \( 0 \) and, assuming the Parkin-Shanks Conjecture, \( 1/2 \), respectively.  Then letting \( \Sigma \) denote the set containing \( 0 \) and all positive integers \( n \) such that \( \sigma (n) \equiv 1 \pmod{2} \), we ask: what are \( \delta(\Sigma) \) and \( \delta(\overline{\Sigma}) \)?

Throughout this sequel, we use the following notation.

\begin{definition} For any set containing nonnegative integers \( F \), we write \( F(q) \) for the ordinary generating function over \( \FF_2 \) of (the indicator function of) \( F \). In other words,
\[
F(q) = \sum_{f \in F} q^f.
\]
Furthermore, for a set of nonnegative integers \( F \), we write \( F^k \) for the set of indices of nonvanishing monomials in \( F(q)^k \).
\end{definition}

\begin{definition} For any set \( F \subseteq \NN \) with \( 0 \in F \), let \( \overline{F} \) be the unique set obtained from \( F(q) \) by defining \( F(q) \overline{F}(q) = 1 \).
\end{definition}

\begin{definition} For any set \( F \subseteq \NN \), let the set of even elements of \( F \) be denoted \( F_e \) and the odd elements \( F_o \), so that \( F(q) = F_e(q) + F_o(q) \).
\end{definition} \text{}
\\
We repeatedly employ the following result, sometimes known as the ``Children's Binomial Theorem''.

\begin{theorem} For any \( f \), \( g \in \FF_2[[q]] \), \( (f+g)^2 \equiv f^2+g^2 \).
\end{theorem}

\section{The Sum of Divisors Function} \label{SODF}

\begin{definition} Let \( \Sigma(q) \) be the binary generating function for \( \sigma(n) \) for nonnegative integers \( n \).  By definition, \( \Sigma \) is the set containing nonnegative integers \( n \) with \( \sigma(n) \) odd.
\[
\Sigma(q) = \sum_{n = 0}^\infty\sigma(n)q^n
\]
\end{definition}

In order to find the density of \( \Sigma \), we need a description of those integers that have an odd divisor sum.  Let \( n \) be a positive natural number.  Then \( \sigma(n) \) is defined to be the sum of all (positive) divisors of \( n \), including \( n \) itself, i.e.,
\[
\sigma(n) = \sum_{d | n} d.
\]
We can write the prime factorization of \( n \) as
\[
n = \prod_{i=1}^k p_i^{e_i}
\]
where the \( p_i \) are distinct primes and \( e_i \) is a positive integer.  Clearly, \( d \in \NN \) divides \( n \) if and only if it can be written in the form
\[
d = \prod_{i=1}^k p_i^{f_i}
\]
where \( 0 \leq f_i \leq e_i \) for each \( i \).  Therefore,
\begin{align}
\sigma(n) &= \sum_{d | n} d \nonumber \\
&= \sum_{f_1=0}^{e_1} \cdots \sum_{f_k=0}^{e_k} \prod_{i=1}^k p_i^{f_i}  \nonumber \\
&= \prod_{i=1}^k \sum_{f_i = 0}^{e_i} p_i^{f_i}. \label{eq1}
\end{align}
This product is odd precisely when all of its factors \( \sum_{f_i} p_i^{f_i} \)are odd.

\begin{lemma} \label{lem1} For an odd prime \( p \), the quantity \( \sum_{f = 0}^{e} p^{f} \) is odd precisely when \( e \) is even.
\end{lemma}
\begin{proof} If \( e \) is even, we may write
\[
\sum_{f = 0}^e p^f = (1 + p)(1 + p^2 + p^4 + \cdots + p^{e-2}) + p^e.
\]
Since \( p \) is odd, \( 1+p \) is even, so the first summand is even.  The second summand, being a nonnegative power of an odd integer, is odd.  Therefore, the sum is odd.  If \( e \) is odd, we may write
\[
\sum_{f = 0}^e p^f = (1 + p)(1 + p^2 + p^4 + \cdots + p^{e-1}).
\]
Again, \( 1+p \) is even, so the sum is even.
\end{proof}

\begin{lemma} \label{lem2} The quantity \( \sum_{k=0}^r 2^r \) is odd for any natural number \( r \).
\end{lemma}
\begin{proof} This follows immediately from the observation that
\[
\sum_{k=0}^r 2^r = 2^{r+1} - 1.
\]
\end{proof}

\begin{theorem} \label{thm1} The quantity \( \sigma(n) \) is odd if and only if the odd part of \( n \) is a square, i.e., if \( n = 2^r \cdot p_1^{e_1} \cdots p_k^{e_k} \) is the prime factorization of \( n \), then \( e_i \) is even for each \( i \), \( 1 \leq i \leq k \).
\end{theorem}
\begin{proof}
By (\ref{eq1}), we may write
\[
\sigma(n) = \left ( \sum_{j=0}^r 2^j \right ) \cdot \left (\prod_{i=1}^k\sum_{j=0}^{e_i} p_i^j \right ).
\]
The left-hand factor is always odd, by Lemma \ref{lem2}.  The right-hand factor is odd if and only if {\em all} of the sums \( \sum_{j} p_i^j \) are odd.  By Lemma \ref{lem1}, this occurs if and only if each of the \( e_i \) is even.  Therefore, \( \sigma(n) \) is even if and only if its odd part is a square.
\end{proof}

\begin{definition} Let \( S(q) \) be the power series with positive squares as the exponents of \( q \):
\[
S(q) = \sum_{n = 1}^\infty q^{n^2}
\]
\end{definition}

\begin{lemma} \label{lem3} \( \Sigma(q) = 1 + S(q) + S(q)^2 \)
\end{lemma}
\begin{proof} The power series \( S(q) \) has a nonzero coefficient for \( q^n \) if and only if \( n \) is a positive integer of the form \( k^2 \).  By the Children's Binomial Theorem, \( S(q)^2 \) has exactly the positive integers of the form \( 2k^2 \) as exponents of \( q \).  By adding \( 1 \), \( S(q) \), and \( S(q)^2 \) together, we obtain a power series whose nonzero monomials are all the nonnegative integer powers of \( q \) of the form \( k^2 \) or \( 2k^2 \).  By Theorem \ref{thm1}, \( \Sigma(q) \) has a nonzero coefficient of \( q^n \) if and only if the odd part of \( n \) is a square, i.e., \( n \) is any nonnegative integer of the form \( k^2 \) or \( 2k^2 \).  The claimed equality follows immediately.
\end{proof}

\begin{cor} \( \delta(\Sigma) = 0 \).
\end{cor}
\begin{proof} The relative density of \( \Sigma \), \( \delta(\Sigma,n) \), is just the number of nonnegative integers less than or equal to \( n \) that are either a square or twice a square, divided by \( n+1 \), so
\[
\delta(\Sigma,n) = \frac{1+\left\lfloor \sqrt{n}\right\rfloor+\left\lfloor \sqrt{n/2} \right\rfloor}{n+1}
\]
Taking the limit as \( n \) tends to infinity yields
\[
\delta(\Sigma) = 0.
\]
\end{proof}

\section{The Reciprocal} \label{Rec}

\begin{definition} Let \( D \) denote the odd (positive) squares, i.e.,
\[
D(q) = \sum_{n  = 0}^\infty q^{(2n+1)^2}
\]
\end{definition}
\begin{lemma} \label{lem4} \( S(q) = \sum_{n = 0}^\infty D(q)^{4^n} \).
\end{lemma}
\begin{proof} Consider the set of all odd squares, \( D \). If we quadruple the members of \( D \), we obtain those even squares divisible by \( 4 \), but not divisible by \( 16 \).  If we quadruple again, we obtain the even squares divisible by \( 16 \), but not \( 64 \).  Applying the Children's Binomial Theorem,
\[
S(q) = \sum_{n = 0}^\infty D(q)^{4^n}.
\]
\end{proof}

\begin{cor} \label{cor1} \( \Sigma(q) = 1 + \sum_{n = 0}^\infty D(q)^{2^n} \)
\end{cor}
\begin{proof} Substituting the equation in the statement of Lemma \ref{lem4} into the expression in Lemma \ref{lem3} and applying the Children's Binomial Theorem once again, we obtain the desired statement.
\end{proof}

The next results will allow us to decompose \( \overline{\Sigma} \) into congruence classes, so that we are able to analyze \( \delta\left(\overline{\Sigma}\right) \) ``piecewise''.

\begin{lemma} \label{lem5}
\[
\overline{\Sigma} (q) = \sum_{n=0}^\infty D (q)^{2^n-1}
\]
\end{lemma}
\begin{proof} We begin by squaring \( \Sigma(q) \) and rewriting:
\[
\Sigma(q)^2 = \left(1 +\sum_{n = 0}^\infty D(q)^{2^n}\right)^2 = 1 + \sum_{n = 1}^\infty D(q)^{2^n}.
\]
If we add \( D(q) \), we have
\[
\Sigma(q)^2 + D(q) = \Sigma(q),
\]
or
\[
D(q) = \Sigma(q) + \Sigma(q)^2,
\]
which we can divide by \( D(q) \Sigma(q) \) to obtain
\begin{equation} \label{eq3}
\overline{\Sigma}(q) = \frac{1+\Sigma(q)}{D(q)} = \sum_{n = 0}^\infty D(q)^{2^n-1}
\end{equation}
by Corollary \ref{cor1}.
\end{proof}

\begin{definition}
Let \( \overline{\Sigma}_k \) denote the subset of \( \overline{\Sigma} \) of integers congruent to \( k \pmod 8 \), i.e., \( \overline{\Sigma}_k = \overline{\Sigma} \cap (8 \ZZ + k) \).
\end{definition}
\begin{lemma} \label{subsets} Using the above definition, the following hold:

\begin{enumerate}[i.]
\item \( \overline{\Sigma}_0(q) = 1 \), i.e., \( \overline{\Sigma}_{0} = \left\{0\right\} \).
\item \( \overline{\Sigma}_1(q) = D(q) \), i.e., \( \overline{\Sigma}_{1} = \left\{\left(2k+1\right)^2 \mid k\in \NN, k \geq 0\right\} \).
\item \( \overline{\Sigma}_3(q) = D(q)^3 \).
\item \( \overline{\Sigma}_7(q) = D(q)^7 + D(q)^{15} + D(q)^{31} + \cdots  \).
\end{enumerate}
\end{lemma}

\begin{proof} The proof proceeds as follows.  For any \( F(q) \in \FF_2[[q]] \), \( F(q)^k \) is the power series whose exponents are those integers that can be represented as a sum of \( k \) of \( F(q) \)'s monomial exponents in an odd number of ways.  Note that the exponents of \( q \) accompanying nonzero coefficients in \( D(q) \) are congruent to \( 1 \pmod 8 \).  Therefore, when \( D(q) \) is raised to a power congruent to \( k \pmod 8 \), the exponents of the resulting series are all congruent to \( k \pmod 8 \).  Proceeding from (\ref{eq3}), we may write
\[
\overline{\Sigma} = \overline{\Sigma}_0 \cup \overline{\Sigma}_1 \cup \overline{\Sigma}_3 \cup \overline{\Sigma}_7,
\]
because the powers of \( D(q) \) on the right-hand side of (\ref{eq3}) are congruent only to \( 0 \), \( 1 \), \( 3 \), and \( 7\) \( \pmod 8 \).  Indeed, by examining those exponents which appear in the terms of (\ref{eq3}), it is straightforward to see that the above lemma holds.
\end{proof}
Our next steps concern further classifications of \( \overline{\Sigma}_3 \) and \( \overline{\Sigma}_7 \). We put \( \overline{\Sigma}_3 \) aside until the end of this section. Presently, we provide the following definition and lemma.

\begin{definition} Let \( \Delta \) denote the set of triangular numbers, i.e.,
\[
\Delta(q) = \sum_{n = 0}^\infty q^{n(n+1)/2}.
\]
\end{definition}

\begin{lemma} \label{lem6} The following identities hold:
\begin{enumerate}[i.]
\item \label{eq5} \( \overline{\Sigma}(q) -1 -D(q) -D(q)^3 = D(q)^7\overline{\Sigma}(q)^8 \)
\item \label{eq6} \( q\Delta(q)^8 = D(q) \)

\end{enumerate}
\end{lemma}
\begin{proof}  By Corollary \ref{cor1},
\[
\overline{\Sigma}(q)= \sum_{n = 0}^\infty D(q)^{2^n-1}.
\]
Therefore,
\begin{align*}
\overline{\Sigma}(q) -1 -D(q) -D(q)^3 &= \sum_{n = 0}^\infty D(q)^{2^{n+3}-1} \\
&= D(q)^7 \sum_{n = 0}^\infty D(q)^{2^{n+3}-8} \\
&= D(q)^7 \overline{\Sigma}(q)^8,
\end{align*}
which is claim (\ref{eq5}).  Now, observe that multiplying a triangular number \( n(n+1)/2 \), \( n \geq 0 \), by \( 8 \) and adding \( 1 \) yields an odd positive square, and in fact, every odd positive square can be uniquely obtained in this manner.  Therefore,
\[
q\Delta(q)^8 = q \left(\sum_{n = 0}^\infty q^{\frac{n(n+1)}{2}}\right)^8 = D(q),
\]
yielding claim (\ref{eq6}).
\end{proof}

The following result is \cite[Theorem 357]{HW} reduced modulo 2.

\begin{theorem}  \label{delta identity}
\[
\Delta(q) = \prod_{ n \geq 1} \left(1+q^n\right)^3
\]
\end{theorem}

We believe the following theorem, though strictly speaking is not needed in its full generality for our main result, holds some independent interest.

\begin{theorem} Let \( k \in \ZZ \) and \( G(q) = \prod_{n \geq 1} (1 + q^n) \).  If \( k \) is odd, then
\[
\Sigma(q) G^k(q) = G^{k}_e(q).
\]
In particular, since \( G^3(q) = \Delta(q) \) and \( G^{-1}(q) = P(q) \), we have that \( \Delta(q)\Sigma(q) = \Delta_e(q) \) and \( P(q) \Sigma(q) = P_e(q) \). 
\end{theorem}
\begin{proof}
Over \( \FF_2 \), the derivative with respect to \( q \) of \( q^n \) is \( 0 \) if \( n \) is even and \( q^{n-1} \) if \( n \) is odd.  Taking the derivative of the expression \( G^k(q) = \prod_{n \geq 1} (1 + q^n)^{k} \) where \( k \) is odd, we see that
\[
\frac{d}{dq}(G^k(q)) = G^k(q) \sum_{ n \geq 1} \frac{nq^{n-1}}{1+q^n},
\]
which simplifies to
\[
\frac{G^k_o(q)}{q} = G^k(q) \sum_{ n \geq 0} \frac{q^{2n}}{1+q^{2n+1}}.
\]
This may be rewritten as
\[
\frac{G^k_o(q)}{G^k(q)} = \sum_{ n \geq 0} \frac{q^{2n+1}}{1+q^{2n+1}}.
\]
If we add \( G^k(q)/G^k(q) \) to the left and \( 1 \) to the right, we arrive at
\begin{equation} \label{delta fraction 2}
\frac{G^k_e(q)}{G^k(q)} = 1+\sum_{ n \geq 0} \frac{q^{2n+1}}{1+q^{2n+1}}.
\end{equation}

The right-hand side of (\ref{delta fraction 2}) has monomial terms \( 1 \) and \(q^n \) for all positive integers \( n \) which are divisible by exactly an odd number of odd numbers.  Note that, if \( r \) is the largest integer so that \( 2^r | n \), then
\begin{align*}
\sigma(n) &= (2^{r+1}-1) \sum_{2 \nmid d | n} d\\
&= \sum_{2 \nmid d | n} 1 \\
&= \begin{cases} 1 & \text{ if } n \text{ has an odd number of odd divisors;} \\ 0 & \text{ otherwise.} \end{cases}
\end{align*}
The desired conclusion is then the reciprocal of (\ref{delta fraction 2}).
\end{proof}

\begin{cor} \label{delta fraction 1} \(  \overline{\Sigma}(q) = \Delta(q)/\Delta_e(q) \).
\end{cor}

\begin{definition} \label{v} Let \( V(q) \) be the power series such that \( \overline{\Sigma}_7(q) = q^7V(q)^8 \).
\end{definition}

This definition is meaningful because \( \overline{\Sigma}_7(q) \) has only monomial terms of the form \( q^k \) where \( k \equiv 7 \pmod 8 \).

\begin{definition} Let \( T(q) = \frac{\Delta(q)^4}{\sqrt{\Delta_e(q)}} \).
\end{definition}
The square root and the fraction make sense because \( \Delta_e(q) \) has only even exponents and \( \sqrt{\Delta_e(q)} \) has a \( 1 \) term since \( \Delta_e(q) \) has a \( 1 \) term (\( 0 \) is an even triangular number).

\begin{theorem} \label{final equation} \( \overline{\Sigma}_7(q) = q^7T(q)^{16} \).
\end{theorem}

\begin{proof} We begin by combining Lemma \ref{lem5} and Lemma \ref{subsets} (iv):

\begin{equation}\label{eq4}
\overline{\Sigma}(q) = 1 + D(q) + D(q)^3 + \overline{\Sigma}_7(q).
\end{equation}
Applying part (\ref{eq5}) of Lemma \ref{lem6} to Definition \ref{v},
\[
D(q)^7\overline{\Sigma}(q)^8 = q^7V(q)^{8},
\]
which, using part (\ref{eq6}) of Lemma \ref{lem6}, becomes
\[
\left(q\Delta(q)^8\right)^7\overline{\Sigma}(q)^8 = q^7V(q)^{8}.
\]
Finally, a bit of simplification reveals that
\begin{equation} \label{eqk}
V(q) = \Delta(q)^7\overline{\Sigma}(q).
\end{equation}
If we substitute Lemma \ref{delta fraction 1} into (\ref{eqk}),
\[
V(q) = \frac{\Delta(q)^8}{\Delta_e(q)},
\]
which is exactly \( T(q)^2 \) on the right-hand side.  Substituting into Definition \ref{v}, we have our proof.
\end{proof}

Finally, we handle the case of exponents which are \(3 \pmod{8}\).  It is straightforward to see that these terms make a vanishing contribution to the density of \(\overline{\Sigma}\).

\begin{prop} \label{prop:3mod8} \( \delta(\overline{\Sigma}_3) = 0 \).
\end{prop}
The interested reader can also arrive at this conclusion by demonstrating that \( \overline{\Sigma}_3 = \left\{p^{e}k^{2} \mid p\text{ prime}, p \equiv 3 \pmod{8}, e \equiv 1 \pmod{4}, k\text{ odd}, p \nmid k \right\} \).
\begin{proof} By Lemma \ref{subsets}, \( \overline{\Sigma}_3(q) = D(q)D(q)^2 \).  Using the Children's Binomial Theorem, we see that \( n \in \overline{\Sigma}_3 \) if and only if the number of representations of \( n \) as \( a^2 + 2b^2 \), with \( a \) and  \( b \) positive is odd. The proposition then follows from the well-known fact that the theta series of a binary quadratic form is lacunary.  (See, for example, \cite{F}.)
\end{proof}

\begin{theorem} \label{thm:main} \( \delta(\overline{\Sigma}) = \delta(\overline{\Sigma}_7) \)
\end{theorem}
\begin{proof}[Proof of Theorem \ref{thm:main}] To show that \( \delta(\overline{\Sigma}) = \delta(\overline{\Sigma}_7) \), we recall that
\[
\overline{\Sigma} = \overline{\Sigma}_0 \cup \overline{\Sigma}_1 \cup \overline{\Sigma}_3 \cup \overline{\Sigma}_7.
\]
Since each \(\overline{\Sigma}_k \) is disjoint, \( \delta(\overline{\Sigma}) = \delta(\overline{\Sigma}_0) + \delta(\overline{\Sigma}_1) + \delta(\overline{\Sigma}_3) + \delta(\overline{\Sigma}_7) \).  Because \( \overline{\Sigma}_0 = \left\{0\right\} \) and \( \overline{\Sigma}_1 = \left\{\left(2k+1\right)^2 \mid k\in \NN, k \geq 0\right\} \), \( \delta(\overline{\Sigma}_0) + \delta(\overline{\Sigma}_1) = 0 \).  By Proposition \ref{prop:3mod8}, \( \delta(\overline{\Sigma}_3) = 0 \).
\end{proof}
\begin{cor} \( 0 \leq \delta(\overline{\Sigma}) \leq 1/16 \).
\end{cor}
\begin{proof} By Theorem \ref{final equation}, \( \overline{\Sigma}_7(q)=q^7T(q)^{16} \).  Applying the Children's Binomial Theorem yet again, if \( n \) is a monomial exponent of \( q \) on the right-hand side of the preceding equation, then \( n \equiv 7 \pmod{16} \). (Equivalently, for \( n \equiv 15 \pmod{16} \), \( n \notin \overline{\Sigma}_7 \).)  Thus, \( 0 \leq \delta\left(\overline{\Sigma}_7\right) \leq 1/16 \) and by Theorem \ref{thm:main}, \( 0 \leq \delta\left(\overline{\Sigma}\right) \leq 1/16 \).
\end{proof}

Numerical evidence strongly suggests the following conjecture.

\begin{conj} \( \delta(\overline{\Sigma}) = 1/32 \).
\end{conj}

\section{Appendix}

We conclude the paper with a few observations about the indices of elements of \( \Sigma \) corresponding to certain well-studied integer sequences.

\begin{definition} \label{c} Let \( c(n) = \lfloor \sqrt{n} \rfloor + \lfloor \sqrt{n/2} \rfloor \), the number of positive integers of the form \(k^2\) or \(2k^2\) less than or equal to \(n\).
\end{definition}

\begin{definition} Let \( \{\varsigma_n \}_{n =1}^\infty \) be the monotone increasing sequence comprised of all positive elements of \( \Sigma \).
\end{definition}

\begin{prop} \label{indices} For all \( n \geq 1 \), \( c(\varsigma_{n}) = n \).
\end{prop}

\begin{proof} For any \( n \geq 1 \), \(n\) is the number of elements of \( \{\varsigma_{n}\}_{n = 1}^\infty \) less than or equal to \( \varsigma_{n} \).  Since \( \{\varsigma_{n}\}_{n = 1}^\infty \) is each positive integer of the form \( k^2 \) or \( 2k^2 \) in monotone increasing order, \( c( \varsigma_{n}) = n \).
\end{proof}

\begin{definition} Below, we define six (non-homogeneous) Beatty sequences for \( k \geq 1 \):
\end{definition}

\begin{enumerate}[i.]
\item Let \( w_k \) be the \(k\)-th winning positions in the 2-Wythoff game (OEIS A001954).
\item Let \( \alpha_{k} = \lfloor k(2+\sqrt{2}) \rfloor \) (OEIS A001952).
\item Let \( \beta_{k} = \lfloor k(2+\sqrt{2})/2 \rfloor \) (OEIS A003152).
\item Let \( \gamma_{k} = \lfloor (k-1/2)(2+2\sqrt{2}) \rfloor \) (OEIS A215247).
\item Let \( \delta_{k} = \lfloor k(2+2\sqrt{2}) \rfloor \) (OEIS A197878).
\item Let \( \epsilon_{k} = \lfloor k (1+\sqrt{2}) \rfloor \) (OEIS A003151).
\end{enumerate}

The following is a result in the 2-Wythoff winning positions \cite{C}.

\begin{theorem} \label{thm:w} \( w_k = \lfloor (k - 1/2)(2 + \sqrt{2}) \rfloor \).
\end{theorem}

\begin{prop} \label{sigalpbetgamdelprop} Let \( \varsigma_n \), \( \beta_n \), \( \alpha_n \), \( \delta_n \), and \( \gamma_n \) be defined as above.  Then we have the following.
\begin{enumerate}[i.]
\item \label{odd squares} If \( n = (2k-1)^2 \) for some positive integer \( k \), then \( n \) is the \( w_k \)-th term in \( \{\varsigma_{n}\}_{n = 1}^\infty \).
\item \label{even squares} If \( n = 4k^2 \) for some positive integer \( k \), then \( n \) is the \( \alpha_{k} \)-th term in \( \{\varsigma_{n}\}_{n = 1}^\infty \).
\item \label{squares} If \( n = k^2 \) for some positive integer \( k \), then \( n \) is the \( \beta_{k} \)-th term in \( \{\varsigma_{n}\}_{n = 1}^\infty \).
\item \label{twice odd squares} If \( n = 2(2k-1)^2 \) for some positive integer \( k \), then \( n \) is the \( \gamma_{k} \)-th term in \( \{\varsigma_{n}\}_{n = 1}^\infty \).
\item \label{twice even squares} If \( n = 8k^2 \) for some positive integer \( k \), then \( n \) is the \( \delta_{k} \)-th term in \( \{\varsigma_{n}\}_{n = 1}^\infty \).
\item \label{twice squares} If \( n = 2k^2 \) for some positive integer \( k \), then \( n \) is the \( \epsilon_{k} \)-th term in \( \{\varsigma_{n}\}_{n = 1}^\infty \).
\end{enumerate}
\end{prop}

\begin{proof} Let \(n = (2k-1)^2 \) be the \( k \)-th odd square.  By Definition \ref{c},
\begin{align*}
c(n) &= \lfloor \sqrt{(2k-1)^2} \rfloor + \lfloor \sqrt{(2k-1)^2/2} \rfloor \\
&= \lfloor 2k-1 + (2k-1)/\sqrt{2} \rfloor \\
&= \lfloor (2k-1)(1+1/\sqrt{2}) \rfloor \\
&= \lfloor (k-1/2)(2+\sqrt{2}) \rfloor.
\end{align*}
Theorem \ref{thm:w} and Proposition \ref{indices} complete the proof of (\ref{odd squares}).  Now, let \(n = 4k^2 \) be the \( k \)-th positive even square.  By Definition \ref{c},
\begin{align*}
c(n) &= \lfloor \sqrt{4k^2} \rfloor + \lfloor \sqrt{4k^2/2} \rfloor \\
&= \lfloor 2k + 2k/\sqrt{2}) \rfloor \\
&= \lfloor k(2+\sqrt{2}) \rfloor.
\end{align*}
Proposition \ref{indices} completes the proof of (\ref{even squares}).  Let \(n = k^2 \) be the \( k \)-th positive square.  By Definition \ref{c},
\begin{align*}
c(n) &= \lfloor \sqrt{k^2} \rfloor + \lfloor \sqrt{k^2/2} \rfloor \\
&= \lfloor k + k/\sqrt{2}) \rfloor \\
&= \lfloor k(2+\sqrt{2})/2 \rfloor.
\end{align*}
Again, Proposition \ref{indices} completes the proof of (\ref{squares}).  Let \(n = 2(2k-1)^2 \) be the \( k \)-th positive twice odd square.  By Definition \ref{c},
\begin{align*}
c(n) &= \lfloor \sqrt{2(2k-1)^2} \rfloor + \lfloor \sqrt{2(2k-1)^2/2} \rfloor \\
&= \lfloor \sqrt{2}(2k-1) + 2k-1 \rfloor \\
&= \lfloor (2k-1)(1+\sqrt{2}) \rfloor \\
&= \lfloor (k-1/2)(2+2\sqrt{2}) \rfloor.
\end{align*}
Proposition \ref{indices} completes the proof of (\ref{twice odd squares}).  Let \(n = 8k^2 \) be the \( k \)-th positive twice even square.  By Definition \ref{c},
\begin{align*}
c(n) &= \lfloor \sqrt{8k^2} \rfloor + \lfloor \sqrt{8k^2/2} \rfloor \\
&= \lfloor 2\sqrt{2}k + 2k \rfloor \\
&= \lfloor k(2+2\sqrt{2}) \rfloor.
\end{align*}
Proposition \ref{indices} completes the proof of (\ref{twice even squares}).  Finally, let \(n = 2k^2 \) be the \( k \)-th positive twice-square.  By Definition \ref{c},
\begin{align*}
c(n) &= \lfloor \sqrt{2k^2} \rfloor + \lfloor \sqrt{2k^2/2} \rfloor \\
&= \lfloor k\sqrt{2} + k \rfloor \\
&= \lfloor k(1+\sqrt{2}) \rfloor.
\end{align*}
Proposition \ref{indices} completes the proof of (\ref{twice squares}).
\end{proof}

Proposition \ref{sigalpbetgamdelprop}(\ref{odd squares}) may be interpreted in the following somewhat surprising way.  Let \( \WW \) denote the positive natural numbers, i.e., the ``whole numbers''.  Define a function \( \mathcal{L} : \WW^\WW \rightarrow \WW^\WW \) as follows: given a function \( f \in \WW^\WW \) which takes on infinitely many odd values, let \( \mathcal{L}(f) \) be the function \( g \in \WW^\WW \) so that \( g(k) \) is the \( k \)-th smallest integer \( n \) so that \( f(n) \) is odd, i.e., for \( k \geq 1 \),
\[
g(k) = \min \{ n : f(n)\equiv 1 \!\!\!\! \pmod 2 \text{ and } n > g(k-1) \},
\]
where we take \( g(0) = -\infty \) by convention.  Then \( \mathcal{L}(\mathcal{L}(\sigma)) = c \).

\section{Acknowledgements}
We would like to thank Mr.~Nicholas.~J.~Smith, Dr. Randy.~M.~La Cross, and Dr.~Gary L.~Salazar for their support, as well as the anonymous referee for several helpful comments.

\bigskip
\hrule
\bigskip

\noindent 2010 {\it Mathematics Subject Classification}:
Primary 11B05; Secondary 11P83, 11A25.

\noindent \emph{Keywords: }
Parkin-Shanks conjecture, Sum of divisors function, Reciprocal sets.

\bigskip
\hrule
\bigskip

\noindent (Concerned with OEIS sequences
A000203,
A001952,
A001954,
A003151,
A003152,
A028982,
A052002,
A192628,
A192717,
A192718,
A197878,
A210449,
A210450, and
A215247.)

\bigskip
\hrule
\bigskip

\vspace*{+.1in}
\noindent
Received XXXXXXX;
revised versions received XXXXXXX; XXXXXXX.
Published in {\it Journal of Integer Sequences}, XXXXXXXX.

\bigskip
\hrule
\bigskip

\noindent
Return to
\htmladdnormallink{Journal of Integer Sequences home page}{http://www.cs.uwaterloo.ca/journals/JIS/}.
\vskip .1in

\end{document}